\documentclass[12pt, a4paper]{article}
\usepackage{amsfonts}
\usepackage{amssymb}
\usepackage{bbm}
\usepackage{amscd}
\usepackage{mathrsfs}
\usepackage{amsmath,amssymb, amsthm}
\usepackage{graphicx}
\usepackage{color}
\usepackage{yhmath}
\usepackage{authblk}
\usepackage{subcaption}
\usepackage{indentfirst,latexsym,amssymb}
\usepackage[bookmarks=false]{hyperref}
\usepackage{tikz}

\newtheorem{theorem}{Theorem}[section]

\newtheorem{lem}[theorem]{Lemma}

\newtheorem{coro}[theorem]{Corollary}

\theoremstyle{definition}

\def\00{\mathbf{0}}

\def\Aut{\hbox{\rm Aut}}

\def\Cay{\hbox{\rm Cay}}

\def\Ga{\Gamma}

\def\Si{\Sigma}

\def\a{\alpha}

\def\l{\lambda}

\usepackage{xcolor}
\usepackage[normalem]{ulem}

\begin{document}

\title{Nowhere-zero 3-flows in nilpotently vertex-transitive graphs}
\author[a]{Junyang Zhang}
\author[b]{Sanming Zhou}
\affil[a]{School of Mathematical Sciences, Chongqing Normal University, Chongqing 401331, P. R. China}
\affil[b]{School of Mathematics and Statistics, The University of Melbourne, Parkville, VIC 3010, Australia}
\date{}

\openup 0.5\jot
\maketitle

\renewcommand{\thefootnote}{\fnsymbol{footnote}}
\footnotetext{E-mail address: jyzhang@cqnu.edu.cn (Junyang Zhang), sanming@unimelb.edu.au (Sanming Zhou)}

\begin{abstract}
We prove that every regular graph of valency at least four whose automorphism group contains a nilpotent subgroup acting transitively on the vertex set admits a nowhere-zero $3$-flow.

\end{abstract}

\section{Introduction}
\label{sec:int}

All graphs in this paper are finite and undirected, with loops and parallel edges allowed, and all groups considered are finite. A graph with neither loops nor parallel edges will be referred to as a simple graph. As usual, for a graph $G$, we use $V(G)$ and $E(G)$ to denote its vertex and edge sets, respectively, and for $v \in V(G)$ we use $E(v)$ to denote the set of edges of $G$ incident with $v$. An \emph{orientation} of $G$ is a digraph $D$ with vertex set $V(G)$ obtained from $G$ by endowing each edge of $G$ with one of the two possible directions. Thus, each edge of $G$ is turned into an \emph{arc} of $D$. We use $D^{+}(v)$ to denote the set of edges in $E(v)$ with tail $v$ and $D^{-}(v)$ the set of edges in $E(v)$ with head $v$. Let $\varphi$ be an integer-valued function on $E(G)$ and $k$ a positive integer. Write
\begin{equation*}
  \varphi^{+}(v):=\sum\limits_{e\in D^{+}(v)}\varphi(e)~~\mbox{and}~~
  \varphi^{-}(v):=\sum\limits_{e\in D^{-}(v)}\varphi(e)
\end{equation*}
for every $v\in V(G)$. If $|\varphi(e)|<k$ for all $e\in E(G)$ and $\varphi^{+}(v)=\varphi^{-}(v)$ for all $v\in V(G)$, then the ordered pair $(D, \varphi)$ is called a \emph{$k$-flow} of $G$. A $k$-flow $(D, \varphi)$ is \emph{nowhere-zero} if $\varphi(e)\neq0$ for every $e\in E(G)$. It is well known that, if for some orientation $D$ of $G$ there exists an integer-valued function $\varphi$ on $E(G)$ such that $(D, \varphi)$ is a nowhere-zero $k$-flow of $G$, then for any orientation $D'$ of $G$ there exists an integer-valued function $\varphi'$ on $E(G)$ such that $(D', \varphi')$ is a nowhere-zero $k$-flow of $G$. In this case we say that $G$ \emph{admits} a nowhere-zero $k$-flow.  

In the middle 1950s, Tutte made three celebrated conjectures on integer flows which are still open in their general form. One of them is the following $3$-flow conjecture (see, for example, \cite[Conjecture 1.1.8]{Z1997}): Every $4$-edge-connected graph admits a nowhere-zero $3$-flow. In 1979, Jaeger \cite{J1979} conjectured that there is a positive integer $k$ such that every $k$-edge-connected graph admits a nowhere-zero $3$-flow. Jaeger's conjecture was confirmed by Thomassen \cite{Th2012} who proved that the statement is true when $k=8$. This breakthrough was further improved by Lov\'asz et al. \cite{LTWZ2013} who proved that every $6$-edge-connected graph admits a nowhere-zero $3$-flow. The reader is referred to \cite{Jaeger88, Z1997} for theories and results on Tutte's conjectures and integer flows in general.

Let $G$ and $\tilde{G}$ be graphs. A \emph{homomorphism} $\sigma: G\rightarrow\tilde{G}$ is a mapping from
$V(G)\cup E(G)$ to $V(\tilde{G})\cup E(\tilde{G})$ such that (i) the restriction $\sigma |_{V(G)}$ of $\sigma$ to $V(G)$ is a mapping from $V(G)$ to $V(\tilde{G})$; (ii) the restriction $\sigma |_{E(G)}$ of $\sigma$ to $E(G)$ is a mapping from $E(G)$ to $E(\tilde{G})$ (which may map a non-loop edge to a loop); and (iii) for every $e\in E(G)$, $\sigma$ maps the end-vertices of $e$ onto the end-vertices of $\sigma(e)$. If in addition both $\sigma |_{V(G)}$ and $\sigma |_{E(G)}$ are surjective, then we call $\sigma$ a \emph{surjective homomorphism}. A homomorphism $\sigma:G\rightarrow \tilde{G}$ is called an \emph{isomorphism} if both $\sigma |_{V(G)}$ and $\sigma |_{E(G)}$ are bijective. An isomorphism from $G$ to itself is called an \emph{automorphism} of $G$. All automorphisms of $G$ equipped with the usual composition of permutations form a group. This group, denoted by $\Aut(G)$, is called the \emph{automorphism group} of $G$. A graph $G$ is said to be \emph{$\Gamma$-vertex-transitive} if $\Gamma$ is a subgroup of $\Aut(G)$ acting transitively on $V(G)$, and $G$ is called \emph{vertex-transitive} if it is $\Aut(G)$-vertex-transitive. As a convention, we assume that all vertex-transitive graphs in this paper have no loops. A $\Gamma$-vertex-transitive graph $G$ is called a \emph{Cayley graph} on $\Gamma$ if $\Gamma$ is \emph{regular} on $V(G)$ (that is, not only is $\Ga$ transitive on $V(G)$ but also it is \emph{semiregular} on $V(G)$ in the sense that no nonidentity element of $\Ga$ can fix any vertex of $G$). It is well known that every vertex-transitive graph is regular and every connected vertex-transitive simple graph of valency $k\geq 1$ is $k$-edge-connected \cite{M1971}. Thus, when restricted to the class of vertex-transitive graphs, Tutte's $3$-flow conjecture asserts that every vertex-transitive simple graph of valency at least four admits a nowhere-zero 3-flow. In this regard, Tutte's $3$-flow conjecture has been confirmed for Cayley graphs of valency at least four on abelian groups \cite{PSS2005}, nilpotent groups \cite{NS2009}, generalized dihedral groups \cite{LL2015}, generalized quaternion groups \cite{LL2015}, generalized dicyclic groups \cite{AI2019}, groups of order $pq^2$ for any pair of primes $p, q$ \cite{ZZ2022}, and two families of supersolvable groups \cite{ZZ2021}. In \cite{LZ2016}, it was proved that every graph with valency at least four whose automorphism group contains a solvable subgroup acting transitively on its set of arcs admits a nowhere-zero 3-flow, where an arc is an ordered pair of adjacent vertices. In \cite{ZT2022}, Zhang and Tao proved that Tutte's $3$-conjecture is true for simple graphs with order twice an odd number and valency at least four whose automorphism groups contain a solvable vertex-transitive subgroup which contains a central involution.

A graph is said to be \emph{nilpotently vertex-transitive} if its automorphism group contains a nilpotent subgroup acting transitively on its vertex set. Clearly, every Cayley graph on a nilpotent group is nilpotently vertex-transitive, but the converse is not true. The main result in this paper is as follows.

\begin{theorem}
\label{nil}
Every nilpotently vertex-transitive graph of valency at least four admits a nowhere-zero $3$-flow.
\end{theorem}

Obviously, this result is a generalization of the above-mentioned result of N\'an\'asiov\'a and \v Skoviera \cite[Theorem 4.3]{NS2009} that every Cayley graph of valency at least four on a nilpotent group admits a nowhere-zero $3$-flow.

The rest of the paper is structured as follows. The next section contains some basic definitions and three preliminary results. In Section \ref{sec:cover}, we will discuss multicovers and regular covers of graphs in the context of integer flows, and in Section \ref{sec:tech} we will establish several technical lemmas about vertex-transitive graphs. With these preparations we will present our proof of Theorem \ref{nil} in Section \ref{sec:pf}.

\section{Preparations}
\label{sec:lem}

The reader is referred to \cite{BM2008} and \cite{DM1996, R1995} for graph-theoretical and group-theoretical terminology and notation, respectively. Let $\Ga$ be a group with identity element $1$. An action of $\Ga$ on a set $V$ is a mapping $\Ga \times V \rightarrow V, (\alpha, v) \mapsto \alpha(v)$ such that $1(v) = v$ and $(\beta \alpha)(v) = \beta (\alpha(v))$ for all $\alpha, \beta \in \Ga$ and all $v \in V$. The \emph{stabilizer} of $v$ under the action of $\Ga$ is the subgroup $\Ga_{v} := \{\alpha \in \Ga: \alpha(v) = v\}$ of $\Ga$. The \emph{$\Ga$-orbit} on $V$ containing $v$ is the subset $\{\alpha(v): \alpha \in \Ga\}$ of $V$. It is known that all $\Ga$-orbits on $V$ form a partition of $V$. If there is only one $\Ga$-orbit on $V$, then $\Ga$ is said to be \emph{transitive} on $V$. If only the identity element of $\Ga$ can fix any $v$ in $V$, then $\Ga$ is said to be \emph{semiregular} on $V$. If $\Ga$ is both transitive and semiregular on $V$, then $\Ga$ is said to be \emph{regular} on $V$.

At this point let us reiterate our assumption that all vertex-transitive graphs (in particular, all Cayley graphs) in this paper have no loops, but they may have parallel edges.

In the previous section we defined a Cayley graph on a group $\Ga$ to be a $\Ga$-vertex-transitive graph $G$ such that $\Ga$ acts regularly on $V(G)$. It is well known (see, for example, \cite[Lemma 3.7.2]{GR2004}) that this definition is equivalent to the following one: Let $\Gamma$ be a group with identity element $1$, and let $\Xi$ be an inverse-closed multiset with elements from $\Gamma\setminus\{1\}$ such that for every $\xi \in \Xi$ the multiplicities of $\xi$ and $\xi^{-1}$ in $\Xi$ are equal. Call $\Xi$ a \emph{Cayley multiset} of $\Gamma$ and define $\Cay(\Gamma,\Xi)$ to be the graph with vertex set $\Gamma$ such that for any $\alpha, \beta \in \Ga$ the number of edges joining $\alpha$ and $\beta$ is equal to the multiplicity of $\alpha^{-1}\beta$ in $\Xi$. Note that this graph contains no loops as $1 \notin \Xi$, and it is a simple graph if and only if the multiplicity of each element of $\Xi$ is $1$ (that is, $\Xi$ is a subset of $\Gamma\setminus\{1\}$). It is straightforward to verify that the left regular representation of $\Gamma$ is a subgroup of $\Aut(\Cay(\Gamma,\Xi))$ and therefore $\Cay(\Gamma,\Xi)$ is a Cayley graph on $\Gamma$. Conversely, if $G$ is a Cayley graph on a group $\Gamma$ with no loops, then $\Gamma$ is regular on $V(G)$ and $G$ is isomorphic to $\Cay(\Gamma,\Xi)$, where, for an arbitrarily chosen but fixed vertex $u$ of $G$, the multiplicity in $\Xi$ of every element $\xi \in \Gamma$ is equal to the number of edges of $G$ between $u$ and $\xi(u)$.

We will use the following known result in our proof of Theorem \ref{nil}.

\begin{lem}
{\rm\cite[Theorem 4.3]{NS2009}}
\label{nilCay}
Every Cayley graph on a nilpotent group of valency at least four admits a nowhere-zero $3$-flow.
\end{lem}

Let $n \ge 2$ be an integer. The \emph{circular ladder} with $n$ rungs is defined to be the Cayley graph $CL_n := \Cay(\Gamma,\{\alpha,\alpha^{-1},\beta\})$, where $\Gamma$ is the abelian group with representation $\langle\alpha,\beta\mid \alpha^n=\beta^2=1, \alpha\beta=\beta\alpha\rangle$. Obviously, $\Cay(\Gamma,\{\beta\})$ is a spanning subgraph of $CL_n$, and its edges are called the \emph{rungs} of $CL_n$. The \emph{M\"obius ladder} with $n$ rungs is defined as $M_n := \Cay(\langle \alpha \rangle,\{\alpha,\alpha^{-1},\alpha^{n}\})$, where $\langle \a \rangle$ is the cyclic group of order $2n$ generated by $\alpha$. The edges of $\Cay(\langle \a \rangle,\{\alpha^{n}\})$ are called the \emph{rungs} of $M_n$. A circular or M\"obius ladder is called a \emph{closed ladder}.

In our proof of Theorem \ref{nil}, we will also use the following lemma which is extracted from \cite[Proposition 3.1]{NS2009}. Recall from group theory that an \emph{involution} of a group is an element of order $2$ and a \emph{central involution} is an involution which is commutable with every element of the group.

\begin{lem}
\label{closedl}
Let $\Cay(\Gamma,\Xi)$ be a connected cubic Cayley graph such that $\Xi$ contains a central involution of $\Gamma$. Then $\Gamma$ is a closed ladder.
\end{lem} 

A \emph{$k$-factor} of a graph is a spanning regular subgraph of valency $k$. The edge set of an $1$-factor is called a \emph{perfect matching} of the graph. (By abusing terminology, we also use the term ``perfect matching" to mean the underlying $1$-factor.) An \emph{even graph} is a graph of which each vertex has even valency. A \emph{parity subgraph} of a graph is a spanning subgraph such that the valencies of every vertex in the subgraph and the original graph have the same parity. It is well known \cite[Theorem 21.4]{BM2008} that every even graph admits a nowhere-zero $3$-flow. It is also easy to see that every graph is the edge-disjoint union of any parity subgraph and an even spanning subgraph. Combining these, we see that a graph admits a nowhere-zero $3$-flow if one of its parity subgraphs does.

The following lemma relies on a recent result proved in \cite{ZZ2021}.

\begin{lem}
\label{pm2f}
Let $G$ be a graph of which every vertex is of odd valency. If $G$ has a perfect matching $M$ and two edge-disjoint $2$-factors $A$ and $B$ such that in $A\cup M$ and $B\cup M$ every connected component is a closed ladder whose rung set is contained in $M$, then $G$ admits a nowhere-zero $3$-flow.
\end{lem}
\begin{proof}
Set $H=A\cup B\cup M$ and let $\mathcal{F}$ be the family of connected components of $A\cup M$ and $B\cup M$. By our assumption, $\mathcal{F}$ is a family of closed ladders satisfying the following conditions:
\begin{itemize}
  \item $H=\cup_{L\in\mathcal{F}}L$;
  \item for every $L\in \mathcal{F}$, each edge in $E(L) \cap M$ is a rung of $L$;
  \item each edge in $M$ is a rung of two distinct closed ladders in $\mathcal{F}$.
\end{itemize}
By \cite[Theorem 3.4]{ZZ2021}, $H$ admits a nowhere-zero $3$-flow. Our assumption on $A\cup M$ and $B\cup M$ implies that $M$ is edge-disjoint with both $A$ and $B$. Hence $H$ is a regular graph of valency $5$. Since every vertex of $G$ is of odd valency, it follows that $H$ is a parity subgraph of $G$. Therefore, $G$ admits a nowhere-zero $3$-flow.
\end{proof}

\section{Multicovers and regular covers}
\label{sec:cover}

\subsection{Multicovers}

Let $\sigma: G \rightarrow \tilde{G}$ be a surjective homomorphism. If there exists a positive integer $\ell$ such that $|\sigma^{-1}(\tilde{e})\cap E(v)|=\ell$ for any $\tilde{v}\in V(\tilde{G})$, $v\in\sigma^{-1}(\tilde{v})$ and $\tilde{e}\in E(\tilde{v})$, then we call $\sigma$ a \emph{multicovering projection} and $G$ a \emph{multicover} of $\tilde{G}$. In the special case when $\ell=1$, we call $\sigma$ a \emph{covering projection} and $G$ a \emph{cover} of $\tilde{G}$. Given an orientation $\tilde{D}$ of $\tilde{G}$, we define the \emph{lifting} of $\tilde{D}$ to be the orientation of $G$ in which an edge $e$ of $G$, say, with end-vertices $u$ and $v$, is oriented from $u$ to $v$ precisely when $\sigma(e)$ is oriented from $\sigma(u)$ to $\sigma(v)$ in $\tilde{D}$.

\begin{lem}
\label{lift}
Let $\sigma: G \rightarrow \tilde{G}$ be a multicovering projection and $k\geq2$ an integer. Let $(\tilde{D},\tilde{\varphi})$ be a $k$-flow of $\tilde{G}$ and $D$ the lifting of $\tilde{D}$. Let $\varphi$ be the integer-valued function on $E(G)$ defined by $\varphi(e)=\tilde{\varphi}(\sigma(e))$ for $e\in E(G)$. Then $(D,\varphi)$ is a $k$-flow of $G$.
\end{lem}

\begin{proof}
Write $\tilde{v}=\sigma(v)$ for any $v\in V(G)$. Then
\begin{equation*}
D^{+}(v)=\bigcup_{\tilde{e}\in \tilde{D}^{+}(\tilde{v})}\sigma^{-1}(\tilde{e})\cap E(v)~~\mbox{and}~~ D^{-}(v)=\bigcup_{\tilde{e}\in \tilde{D}^{-}(\tilde{v})}\sigma^{-1}(\tilde{e})\cap E(v).
\end{equation*}
Since $\sigma$ is a multicovering projection from $G$ to $\tilde{G}$, there exists a positive integer $\ell$ such that for any $v\in V(G)$ and $\tilde{e}\in E(\tilde{v})$, we have $|\sigma^{-1}(\tilde{e})\cap E(v)|=\ell$. It follows that
\begin{eqnarray*}
\varphi^{+}(v) & = & \sum\limits_{e\in D^{+}(v)}\varphi(e) \\
& = & \sum\limits_{\tilde{e}\in \tilde{D}^{+}(\tilde{v})}\left(\sum\limits_{e\in\sigma^{-1}(\tilde{e})\cap E(v)}\varphi(e)\right) \\
& = & \ell\sum\limits_{\tilde{e}\in \tilde{D}^{+}(\tilde{v})} \tilde{\varphi}(\tilde{e}) \\
& = & \ell\tilde{\varphi}^{+}(\tilde{v}).
\end{eqnarray*}
Similarly, $\varphi^{-}(v)=\ell\tilde{\varphi}^{-}(\tilde{v})$.
Since $(\tilde{D},\tilde{\varphi})$ is a $k$-flow of $\tilde{G}$, we have $\tilde{\varphi}^{+}(\tilde{v})=\tilde{\varphi}^{-}(\tilde{v})$. Therefore, $\varphi^{+}(v)=\varphi^{-}(v)$ for any $v\in V(G)$ and consequently $(D,\varphi)$ is a $k$-flow of $G$.
\end{proof}

We call the $k$-flow $(D,\varphi)$ in Lemma \ref{lift} the \emph{lifting} of $(\tilde{D},\tilde{\varphi})$. Clearly,
if $(\tilde{D},\tilde{\varphi})$ is nowhere-zero, then so is $(D,\varphi)$. Therefore, Lemma \ref{lift} implies the following corollary.

\begin{coro}{\rm\cite[Lemma 2.1]{LZ2016}}
\label{quotient}
Let $G$ be a multicover of $\tilde{G}$ and $k\geq2$ an integer. Then $G$ admits a nowhere-zero $k$-flow whenever $\tilde{G}$ does.
\end{coro}

\subsection{Regular covers}

Let $G$ be a graph. Let $\Lambda$ be a subgroup of $\Aut(G)$ which acts semiregularly on both $V(G)$ and $E(G)$. (Note that, in general, a subgroup $\Lambda$ of $\Aut(G)$ can be semiregular on $V(G)$ but not on $E(G)$ as an edge can be fixed by an element of $\Lambda$ through swapping its end-vertices. Also, $\Lambda$ can be semiregular on $E(G)$ but not on $V(G)$.) For any $v\in V(G)$ and $e\in E(G)$, let
$$
\Lambda(v):=\{\alpha(v)\mid \alpha\in \Lambda\}~~\mbox{and}~~\Lambda(e):=\{\alpha(e)\mid \alpha\in \Lambda\}
$$
be the $\Lambda$-orbits on $V(G)$ and $E(G)$ containing $v$ and $e$, respectively. Of course, $\{\Lambda(v)\mid v\in V(G)\}$ and $\{\Lambda(e)\mid e\in E(G)\}$ are partitions of $V(G)$ and $E(G)$, respectively. It is readily seen that, for any $e\in E(G)$ with end-vertices $x$ and $y$, the set of end-vertices of the edges in $\Lambda(e)$ is exactly $\Lambda(x) \cup \Lambda(y)$. If we treat each $\Lambda(v)$ as a vertex and each $\Lambda(e)$ as an edge between $\Lambda(x)$ and $\Lambda(y)$, then we obtain a graph $G_\Lambda$ which is the quotient graph of $G$ with respect to the partition $\{\Lambda(v)\mid v\in V(G)\}$ of $V(G)$. Equivalently, we can view $G_\Lambda$ as the graph with vertex set $\{\Lambda(v)\mid v\in V(G)\}$ and edge set $\{\Lambda(e)\mid e\in E(G)\}$ such that $\Lambda(v)$ is incident with $\Lambda(e)$ if and only if $\Lambda(v)$ contains a vertex which is incident with an edge in $\Lambda(e)$. It is obvious that $G_\Lambda$ contains no loops if $\Lambda(v)$ is an independent set of $G$ for every $v\in V(G)$. It can be verified that the mapping
\begin{equation*}
\sigma:V(G)\cup E(G)\rightarrow V(G_\Lambda)\cup E(G_\Lambda)
\end{equation*}
defined by
$$
\sigma(v) = \Lambda(v),\; \sigma(e) = \Lambda(e),\; v\in V(G),\; e\in E(G)
$$
is a covering projection from $G$ to $G_\Lambda$. Hence $G$ is a cover of $G_\Lambda$. Using the language of topological graph theory \cite{GT1987}, $G$ is called a \emph{regular cover} of $G_\Lambda$, and $G_\Lambda$ is called a \emph{regular quotient} of $G$ induced by $\Lambda$.  

Now suppose that $\Gamma$ is a subgroup of $\Aut(G)$ which has a normal subgroup $\Lambda$ acting semiregularly on both $V(G)$ and $E(G)$. Let $\Lambda^{*}$ consist of those elements of $\Gamma$ which leave both $\Lambda(v)$ and $\Lambda(e)$ invariant for every $v\in V(G)$ and $e\in E(G)$. Then $\Lambda^{*}$ is a subgroup of $\Gamma$ containing $\Lambda$. Let $\alpha\in\Lambda^{*}$ and $\gamma\in \Gamma$.  Since
\begin{equation*}
\gamma^{-1}\alpha\gamma(\Lambda(v))
=\gamma^{-1}\alpha\gamma\Lambda\gamma^{-1}(\gamma(v))
=\gamma^{-1}\alpha\Lambda(\gamma(v))=\gamma^{-1}\Lambda(\gamma(v))
=\gamma^{-1}\Lambda\gamma(v)=\Lambda(v)
\end{equation*}
for all $v\in V(G)$ and similarly $\gamma^{-1}\alpha\gamma(\Lambda(e))=\Lambda(e)$
for all $e\in V(G)$,
we have $\gamma^{-1}\alpha\gamma\in\Lambda^{*}$ and thus $\Lambda^{*}$ is normal in $\Gamma$. Therefore, for every $\gamma\Lambda^{*}\in\Gamma/\Lambda^{*}$, the rule
\begin{equation*}
\gamma\Lambda^{*}(\Lambda(v))=\Lambda(\gamma(v))~\mbox{and}~
\gamma\Lambda^{*}(\Lambda(e))=\Lambda(\gamma(e))
\end{equation*}
for $\Lambda(v)\in V(G_{\Lambda})$ and $\Lambda(e)\in E(G_{\Lambda})$ defines a permutation on $V(G_{\Lambda})\cup E(G_{\Lambda})$. One can see that this permutation is not the identity permutation if $\gamma\notin\Lambda^{*}$. This means that $\Gamma/\Lambda^{*}$ acts faithfully on $V(G_{\Lambda})\cup E(G_{\Lambda})$. Since $\Gamma$ preserves the adjacency relation of $G$, $\Gamma/\Lambda^{*}$ preserves the adjacency relation of $G_{\Lambda}$. Thus, $\Gamma/\Lambda^{*}$ can be viewed as a subgroup of $\Aut(G_{\Lambda})$. Moreover, if $\Gamma$ is transitive on $V(G)$, then $\Gamma/\Lambda^{*}$ is transitive on $V(G_{\Lambda})$. Therefore, we have proved the following lemma, where the assumption that $\Lambda(v)$ is an independent set of $G$ ensures that $G_{\Lambda}$ has no loops.

\begin{lem}
\label{vt}
Let $G$ be a $\Gamma$-vertex-transitive graph, where $\Ga$ is a subgroup of $\Aut(G)$. Let $\Lambda$ be a normal subgroup of $\Gamma$ which acts semiregularly on both $V(G)$ and $E(G)$, and let $\Lambda^{*}$ be the normal subgroup of $\Gamma$ consisting of the elements of $\Gamma$ which leave both $\Lambda(v)$ and $\Lambda(e)$ invariant for every $v\in V(G)$ and $e\in E(G)$. If $\Lambda(v)$ is an independent set of $G$ for some $v\in V(G)$, then the regular quotient $G_{\Lambda}$ of $G$ induced by $\Lambda$ is $\Gamma/\Lambda^{*}$-vertex-transitive.
\end{lem}

\section{Technical Lemmas}
\label{sec:tech}

\begin{lem}
\label{free}
Let $G$ be a $\Gamma$-vertex-transitive graph, where $\Ga$ is a subgroup of $\Aut(G)$ acting faithfully on $V(G)$. Suppose that $\Gamma$ has a subgroup $\Lambda$ of which every cyclic subgroup is normal in $\Gamma$. Then $\Lambda$ acts semiregularly on $V(G)$.
\end{lem} 

\begin{proof}
Suppose to the contrary that $\Lambda$ is not semiregular on $V(G)$. Then there exists a nonidentity element $\lambda\in\Lambda$ such that $\lambda(u)=u$ for some $u\in V(G)$. Since $\Ga$ acts faithfully on $V(G)$ and $\lambda$ is not the identity automorphism of $G$, $\lambda(v)\neq v$ for some $v\in V(G)$. Since $G$ is $\Gamma$-vertex-transitive, $v=\alpha(u)$ for some $\alpha\in\Gamma$. Since every cyclic subgroup of $\Lambda$ is normal in $\Gamma$, we have $\alpha^{-1}\lambda\alpha=\lambda^{r}$ for some positive integer $r$. Thus, $\lambda(v)=\lambda\alpha(u)=\alpha\lambda^{r}(u)=\alpha(u)=v$, but this is a contradiction. Therefore, $\Lambda$ acts semiregularly on $V(G)$.
\end{proof}

\begin{lem}
\label{efree}
Let $G$ be a $\Gamma$-vertex-transitive graph, where $\Ga$ is a subgroup of $\Aut(G)$. Let $\Lambda$ be a normal subgroup of $\Gamma$ acting semiregularly on $V(G)$. If at least one $\Lambda$-orbit on $V(G)$ is an independent set of $\Gamma$, then $\Lambda$ acts semiregularly on $E(G)$.
\end{lem}

\begin{proof}
By our assumption we may assume that $\Lambda(u)$ is an independent set of $\Gamma$, where $u \in V(G)$. Consider an arbitrary edge $e$ of $G$, say, with end-vertices $v$ and $w$. Since $G$ is $\Gamma$-vertex-transitive, there exists $\alpha\in\Gamma$ such that $v=\alpha(u)$. Since $\Lambda$ is normal in $\Gamma$, we have
\begin{equation*}
\Lambda(v)=\Lambda(\alpha(u))=\alpha(\alpha^{-1}\Lambda\alpha(u))= \alpha(\Lambda(u)).
\end{equation*}
Since $\alpha$ is an automorphism of $G$ and $\Lambda(u)$ is an independent set of $\Gamma$, it follows that $\Lambda(v)$ is also an independent set of $\Gamma$. Since $v$ and $w$ are adjacent in $G$ and $v \in \Lambda(v)$, we then obtain that $w \notin \Lambda(v)$. It follows that $\lambda(v)\neq w$ for any nonidentity element $\lambda$ of $\Lambda$. Since $\Lambda$ is semiregular on $V(G)$, we also have $\lambda(v) \neq v$. Therefore, $\lambda(e) \neq e$ and consequently $\Lambda$ is semiregular on $E(G)$.
\end{proof}

Since a graph $G$ in this paper may have loops and parallel edges, its automorphism group $\Aut(G)$ as defined in Section \ref{sec:int} may not be faithful on $V(G)$ (as $\Aut(G)$ may contain elements fixing every vertex which interchange two parallel edges and fix all other edges). In other words, the kernel of the action of $\Aut(G)$ on $V(G)$ is not necessarily the trivial subgroup of $\Aut(G)$.

\begin{lem}
\label{faithful}
Let $G$ be a loopless graph. Let $\Theta$ be the kernel of the action of $\Aut(G)$ on $V(G)$. Then $\Aut(G)$ has a subgroup $\Sigma$ acting faithfully on $V(G)$ such that $\Aut(G)=\Sigma\Theta$.
\end{lem}

\begin{proof}
Let $H$ be a simple spanning subgraph of $G$ obtained from $G$ by deleting all but one parallel edges between every pair of adjacent vertices. Since $H$ is spanning, we have $V(H)=V(G)$. For each $e \in E(H)$, let $e=e^{(1)},\ldots,e^{(m(e))}$ be the edges of $G$ with the same end-vertices as $e$, where $m(e)$ is the multiplicity of $e$ in $G$. Then $E(G)=\{e^{(i)}\mid e\in E(H), i=1,\ldots,m(e)\}$. Let $\Gamma$ consist of all automorphisms $\gamma$ of $H$ having the property that, for each $e\in E(H)$, $\gamma(e)$ and $e$ have the same multiplicity in $G$. It is straightforward to verify that $\Gamma$ is a subgroup of $\Aut(H)$. Since $H$ is a simple graph, $\Aut(H)$ acts faithfully on $V(H)$, and hence $\Gamma$ acts faithfully on $V(H)$. For each $\gamma\in \Gamma$, define $\gamma^{*}$ to be the permutation on $V(G)\cup E(G)$ such that $\gamma^{*}(v)=\gamma(v)$ and $\gamma^{*}(e^{(i)})=\gamma(e)^{(i)}$ for all $v\in V(G)$ and $e^{(i)}\in E(G)$. Then $\gamma^{*}$ is an automorphism of $G$ since it maps the end-vertices of every $e^{(i)}\in E(G)$ onto the end-vertices of $\gamma^{*}(e^{(i)})$. Set $\Sigma=\{\gamma^{*}\mid \gamma\in \Gamma\}$. Then $|\Sigma|=|\Gamma|$ and $\Sigma$ is a subgroup of $\Aut(G)$ acting faithfully on $V(G)$. Since $\Theta$ fixes every vertex of $G$, we have $\Sigma\cap\Theta=\{1\}$. Since $\Theta$ is normal in $\Aut(G)$, $\Sigma\Theta$ is a subgroup of $\Aut(G)$.

Define $\tau$ to be a mapping from $V(G)\cup E(G)$ to $V(H)\cup E(H)$ such that $\tau(v)=v$ for all $v\in V(G)$ and $\tau(e^{(i)})=e^{(1)}$ for all $e^{(i)}\in E(G)$. It is obvious that $\tau$ is a surjective homomorphism from $G$ to $H$. For each $\alpha\in \Aut(G)$, define $\alpha'$ to be the permutation on $V(H)\cup E(H)$ such that $\alpha'(v)=\alpha(v)$ and $\alpha'(e^{(1)})=\tau(\alpha(e^{(1)}))$ for all $v\in V(H)$ and $e^{(1)}\in E(H)$. Then $\alpha'$ is an automorphism of $H$ such that, for any $e^{(1)}\in E(H)$, $\alpha'(e^{(1)})$ and $e^{(1)}$ have the same multiplicity in $G$. Hence $\alpha'\in \Gamma$. Since
\begin{equation*}
 (\alpha\beta)'(e^{(1)})=\tau(\alpha\beta(e^{(1)}))
 =\alpha'\tau(\beta(e^{(1)})) =\alpha'\beta'\tau(e^{(1)})
  =\alpha'\beta'(e^{(1)})
\end{equation*}
for all $\alpha, \beta\in \Aut(G)$ and $e^{(1)}\in E(G)$, the map $\psi:\Aut(G)\rightarrow \Gamma, \alpha\mapsto\alpha'$ is a group homomorphism. Since $\alpha'$ is the identity automorphism of $H$ if and only if $\alpha(v)=v$ for all $v\in V(G)$, the kernel of $\psi$ is exactly $\Theta$. By the First Isomorphism Theorem \cite[1.4.3]{R1995}, $\Aut(G)/\Theta$ is isomorphic to a subgroup of $\Gamma$. Thus, $|\Aut(G)/\Theta|\leq|\Gamma|$ and hence $|\Aut(G)|\leq|\Gamma||\Theta|$. Since $|\Sigma|=|\Gamma|$ and $\Sigma\cap\Theta=\{1\}$, we have $|\Sigma\Theta|=|\Sigma||\Theta|=|\Gamma||\Theta|$ and thus $|\Aut(G)| \leq |\Sigma\Theta|$. Since $\Sigma\Theta$ is a subgroup of $\Aut(G)$, it follows that $\Aut(G)=\Sigma\Theta$.
\end{proof}

The following lemma should be known in the literature (at least when the graph involved is simple). Since we are unable to identify a reference to it, we give its proof here for completeness.

\begin{lem}
\label{connect}
Let $G$ be a $\Gamma$-vertex-transitive graph, where $\Ga$ is a subgroup of $\Aut(G)$. Let $u$ be a vertex of $G$, and let $\{\delta_{1}(u),\ldots,\delta_{m}(u)\}$ be the neighborhood of $u$ in $G$, where $\delta_{1},\ldots,\delta_{m}\in\Gamma$. Then $G$ is connected if and only if $\Ga = \langle\Gamma_{u},\delta_{1},\ldots,\delta_{m}\rangle$, where $\Gamma_{u}$ is the stabilizer of $u$ under the action of $\Ga$.
\end{lem}

\begin{proof}
Suppose that $G$ is connected. Then for any $\gamma \in \Gamma$ there exists a path $u=u_0,u_1,\ldots,u_{t-1},u_t=\gamma(u)$
in $G$ between $u$ and $\gamma(u)$. Since $G$ is $\Gamma$-vertex-transitive, for $0 \leq i\leq t$, there exists $\gamma_{i}\in\Gamma$ such that $u_i=\gamma_{i}(u)$, where we choose $\gamma_0 = 1$ and $\gamma_t = \gamma$. Since $\gamma_{i-1}(u)$ and $\gamma_{i}(u)$ are adjacent in $G$ and $\gamma_{i-1}$ is an automorphism of $G$, it follows that $u$ and $(\gamma_{i-1}^{-1}\gamma_{i})(u)$ are adjacent in $G$, for $1\leq i\leq t$. However, the neighborhood of $u$ in $G$ is $\{\delta_{1}(u),\ldots,\delta_{m}(u)\}$. So we have $\gamma_{i-1}^{-1}\gamma_{i} \in \cup_{j=1}^{m}\Gamma_{u}\delta_{j}\Gamma_{u}$, for $1\leq i\leq t$. Since $\gamma = \gamma_0(\gamma_{0}^{-1}\gamma_{1})(\gamma_{1}^{-1}\gamma_{2}) \cdots (\gamma_{t-1}^{-1}\gamma_{t})$, we then have $\gamma \in \langle\Gamma_{u},\delta_{1},\ldots,\delta_{m}\rangle$. Since this holds for every $\gamma \in \Gamma$, we conclude that $\Ga = \langle\Gamma_{u},\delta_{1},\ldots,\delta_{m}\rangle$.

Conversely, suppose that 
$\Ga=\langle\Gamma_{u},\delta_{1},\ldots,\delta_{m}\rangle$. Since $G$ is $\Gamma$-vertex-transitive, any vertex of $G$ is of the form $\gamma(u)$ for some $\gamma\in\Gamma$. Since $\Ga = \langle\Gamma_{u},\delta_{1},\ldots,\delta_{m}\rangle$, we can write $\gamma = \alpha_{1}\alpha_{2} \cdots \alpha_{\ell}$ for some $\alpha_{1}, \alpha_{2}, \ldots, \alpha_{\ell} \in \Gamma_{u} \cup \{\delta_{1},\ldots,\delta_{m}\}$. Note that $\alpha_{i}(u) = u$ if $\alpha_{i} \in \Gamma_{u}$, and $\alpha_{i}(u)$ is adjacent to $u$ if $\alpha_{i} \in \{\delta_{1},\ldots,\delta_{m}\}$. Therefore, in the sequence $u, \alpha_{1}(u), (\alpha_{1}\alpha_{2})(u), \ldots, (\alpha_{1} \alpha_{2} \cdots \alpha_{\ell})(u) = \gamma(u)$, any two consecutive terms are either identical or adjacent in $G$. So this sequence yields a walk from $u$ to $\gamma(u)$, and therefore $G$ contains a path from $u$ to $\gamma(u)$. Since $\gamma(u)$ is an arbitrary vertex of $G$, we conclude that $G$ is connected.
\end{proof}

The Frattini subgroup \cite{R1995} $\Phi(\Gamma)$ of a group $\Gamma$ is the intersection of all maximal subgroups of $\Gamma$. Equivalently, $\Phi(\Gamma)$ is the set of elements $\alpha$ of $\Gamma$ with the property that $\Gamma= \langle \alpha, \Xi \rangle$ always implies $\Gamma= \langle \Xi \rangle$ when $\Xi$ is a subset of $\Gamma$.

\begin{lem}
\label{stabilizer}
Let $G$ be a nilpotently vertex-transitive graph. Let $\Gamma$ be a nilpotent subgroup of $\Aut(G)$ with minimum order acting transitively on $V(G)$. Then the stabilizer $\Gamma_{u}$ of every vertex $u \in V(G)$ is contained in $\Phi(\Gamma)$.
\end{lem}

\begin{proof}
Let $\Sigma$ be an arbitrary maximal subgroup of $\Gamma$. Since $\Ga$ is nilpotent and any maximal subgroup of a nilpotent group is normal (see, for example, \cite[5.2.4]{R1995}), we know that $\Sigma$ is normal in $\Ga$. Therefore, $\Sigma\Gamma_{u}$ is a subgroup of $\Gamma$. As a subgroup of the nilpotent group $\Ga$, $\Si$ is also nilpotent, and so $\Si$ must be intransitive on $V(G)$ by the minimality of $\Gamma$. Note that the orbit of $u$ under $\Sigma$ is the same as that under $\Sigma\Gamma_{u}$. Hence $\Sigma\Gamma_{u}$ is intransitive on $V(G)$ and thus a proper subgroup of $\Gamma$. Since $\Sigma$ is a maximal subgroup of $\Gamma$, we have $\Sigma\Gamma_{u}=\Sigma$ and therefore $\Gamma_{u}$ is contained in $\Sigma$. By the arbitrariness of $\Sigma$, we conclude that $\Gamma_{u}$ is contained in $\Phi(\Gamma)$.
\end{proof}

\begin{lem}
\label{index}
Let $\Gamma$ be a nilpotent group and $\Sigma$ a subgroup of $\Gamma$. Let $\delta$ be an element of $\Gamma$ such that $\delta\Sigma\delta^{-1}\neq\Sigma$. If $\Sigma\delta\Sigma=\Sigma\delta^{-1}\Sigma$, then $|\Sigma:\Sigma\cap\delta\Sigma\delta^{-1}|$ is even.
\end{lem}

\begin{proof}
Since $\Sigma\delta\Sigma=\Sigma\delta^{-1}\Sigma$, we have $\delta^{-1}=\alpha\delta\beta$ for some $\alpha,\beta\in\Sigma$. Hence $(\delta\alpha)^2 = (\delta\beta)^{-1}(\delta\alpha) = \beta^{-1}\alpha \in\Sigma$. Let $s$ be the largest odd divisor of the order of $\delta\alpha$ and set $\mu=(\delta\alpha)^s$. Then $\mu$ is contained in the Sylow $2$-subgroup of $\Gamma$. Since $\Gamma$ is a nilpotent group, we have $\mu\gamma=\gamma\mu$ for every element $\gamma$ of $\Gamma$ with odd order. Therefore, every element of $\Sigma$ with odd order is contained in $\mu\Sigma\mu^{-1}$. So $|\Sigma:\Sigma\cap\mu\Sigma\mu^{-1}|$ is a power of $2$. Since $(\delta\alpha)^2 \in\Sigma$ and $s-1$ is even, we have $(\delta\alpha)^{s-1} \in \Sigma$ and hence $\alpha(\delta\alpha)^{s-1}\in \Sigma$. This together with $\mu=(\delta\alpha)^s=\delta\alpha(\delta\alpha)^{s-1}$ implies that $\mu\Sigma\mu^{-1}=\delta\Sigma\delta^{-1}$. Therefore, $|\Sigma:\Sigma\cap\delta\Sigma\delta^{-1}|$ is a power of $2$. Since $\delta\Sigma\delta^{-1}\neq\Sigma$, $\Sigma\cap\delta\Sigma\delta^{-1}$ is a proper subgroup of $\Sigma$ and hence $|\Sigma:\Sigma\cap\delta\Sigma\delta^{-1}| > 1$. It follows that $|\Sigma:\Sigma\cap\delta\Sigma\delta^{-1}|$ is even.
\end{proof}

Let $G$ be a $\Gamma$-vertex-transitive graph, $u$ a vertex of $G$, and $\alpha$ an element of $\Gamma$. Define $G_{u, \alpha}$ to be the spanning subgraph of $G$ whose edges are those edges of $G$ joining $\gamma(u)$ and $(\gamma\beta\alpha)(u)$, for $\gamma\in \Gamma$ and $\beta\in \Gamma_u$. From this definition it follows that $\Gamma$ preserves the incidence structure of $G_{u,\alpha}$. Moreover, if the action of $\Gamma$ on $V(G)\cup E(G_{u,\alpha})$ is faithful, then $G_{u,\alpha}$ is a $\Gamma$-vertex-transitive graph. We call $G_{u,\alpha}$ the spanning subgraph of $G$ \emph{generated by $\alpha$} with reference vertex $u$. Note that $G_{u,\alpha}$ is not an edgeless graph if and only if $\alpha(u)$ is a neighbor of $u$. Note also that, for any vertex $u'$ of $G$, say, $u' = \l(u)$, where $\l \in \Ga$, the subgraph $G_{u', \alpha'}$ generated by $\alpha' = \l \alpha \l^{-1}$ with reference vertex $u'$ is equal to $G_{u, \alpha}$. Therefore, the family of spanning subgraphs $\{G_{u,\alpha} \mid \alpha \in \Ga\}$ of $G$ is independent of the choice of the reference vertex $u$.

\begin{lem}
\label{generating}
Let $G$ be a $\Gamma$-vertex-transitive graph, $u$ a vertex of $G$, and $\alpha$ an element of $\Gamma$, where $\Ga$ is a subgroup of $\Aut(G)$. If $\alpha(u)$ is a neighbor of $u$ in $G$, then the number of neighbors of $u$ in $G_{u, \alpha}$ is $|\Gamma_u:\Gamma_u\cap\alpha\Gamma_u\alpha^{-1}|$ if $\alpha^{-1} \in \Gamma_{u}\alpha\Gamma_{u}$ and $2|\Gamma_u:\Gamma_u\cap\alpha\Gamma_u\alpha^{-1}|$ if $\alpha^{-1} \notin \Gamma_{u}\alpha\Gamma_{u}$.
\end{lem}

\begin{proof}
Denote by $N(u)$ the neighborhood of $u$ in $G_{u,\alpha}$. Since $\alpha(u)\in N(u)$, by the definition of $G_{u,\alpha}$, a vertex $w$ is in $N(u)$ if and only if $w=\mu(u)$ for some $\mu\in(\Gamma_u\alpha\Gamma_u)\cup(\Gamma_u\alpha^{-1}\Gamma_u)$. Since for any $\gamma,\gamma' \in\Gamma$, $\gamma(u)=\gamma'(u)$ if and only if $\gamma\Gamma_u=\gamma'\Gamma_u$, it follows that $|N(u)|$ is equal to the number of left cosets of $\Gamma_u$ contained in $(\Gamma_u\alpha\Gamma_u)\cup(\Gamma_u\alpha^{-1}\Gamma_u)$. Therefore, $|N(u)|=|\Gamma_u:\Gamma_u\cap\alpha\Gamma_u\alpha^{-1}|$ if $\Gamma_u\alpha\Gamma_u=\Gamma_u\alpha^{-1}\Gamma_u$, and $|N(u)|=2|\Gamma_u:\Gamma_u\cap\alpha\Gamma_u\alpha^{-1}|$ if $\Gamma_u\alpha\Gamma_u\neq\Gamma_u\alpha^{-1}\Gamma_u$.
\end{proof}

Note that the valency of $G_{u,\alpha}$ can be strictly larger than the number of neighbors of $u$ in $G_{u,\alpha}$ as $G_{u,\alpha}$ may contain parallel edges.

As usual, for a graph $G$ and a set $E \subseteq E(G)$, denote by $G - E$ the spanning subgraph of $G$ obtained from $G$ by deleting all edges in $E$.

\begin{lem}
\label{2factor}
Let $G$ be a $\Gamma$-vertex-transitive graph of odd valency at least $5$, where $\Ga$ is a subgroup of $\Aut(G)$. Suppose that $\Gamma$ contains a central involution $\lambda$ such that, for some vertex $u \in V(G)$, $G_{u,\lambda}$ is a perfect matching of $G$ and $G-E(G_{u,\lambda})$ has two edge-disjoint $2$-factors preserved by $\Gamma$. Then $G$ admits a nowhere-zero $3$-flow.
\end{lem}

\begin{proof}
Let $A$ and $B$ be edge-disjoint $2$-factors of $G-E(G_{u,\lambda})$ preserved by $\Gamma$. Then each of them is a union of cycles with the same length (here we treat the graph with two parallel edges between two vertices as a cycle of length $2$).

If $A$ has parallel edges, then each connected component of $A$ is a $2$-cycle. Since $G_{u,\lambda}$ and $A$ have no common edge, each connected component of $A\cup G_{u,\lambda}$ is isomorphic to the graph constructed from a cycle of even length by doubling every edge of one of its two perfect matchings. It follows that $A\cup G_{u,\lambda}$ is a cubic bipartite graph and therefore admits a nowhere-zero $3$-flow. Since $A\cup G_{u,\lambda}$ is a parity subgraph of $G$, we obtain that $G$ admits a nowhere-zero $3$-flow. Similarly, if $B$ has parallel edges, then $G$ admits a nowhere-zero $3$-flow.

From now on we assume that neither $A$ nor $B$ has parallel edges. Let $C$ be an arbitrary connected component of $A$. Then $C$ is a cycle of length $s$ for some integer $s \ge 3$. So $\Aut(C)$ is the dihedral group of order $2s$ of which the cyclic subgroup of order $s$ acts regularly on $V(C)$. If $s$ is odd, then $\Aut(C)$ contains no central involution. If $s$ is even, then $\Aut(C)$ contains a unique central involution which is contained in any subgroup acting regularly on $V(C)$. Let $\Sigma$ be the set of elements of $\Gamma$ that preserve $C$ as a graph. Then $\Sigma$ is a subgroup of $\Gamma$. Since $A$ is preserved by $\Gamma$ and  $\Gamma$ acts transitively on $V(G)=V(A)$, $\Sigma$ acts transitively on $V(C)$. Let $\Theta$ be the kernel of the action of $\Sigma$ on $V(C)$. Then $\Theta$ is normal in $\Sigma$ and $\Sigma/\Theta$ is isomorphic to a vertex-transitive subgroup of $\Aut(C)$. Therefore, $\Sigma$ has a subgroup $\Pi$ such that $\Pi/\Theta$ acts regularly on $V(C)$. Let $M_{C}$ be the set of edges of $G_{u,\lambda}$ with at least one end-vertex in $V(C)$. Take an arbitrary edge $e$ in $M_C$ and let $u$ be an end-vertex of $e$ contained in $V(C)$. Then $\lambda(u)$ is an end-vertex of $e$ and is contained in $V(\lambda(C))$. Set $H=(C\cup\lambda(C)) + M_C$ to be the graph obtained by adding all edges of $M_C$ to $C\cup\lambda(C)$. Then $V(H)=V(C\cup\lambda(C))$. Since $G_{u,\lambda}$ is a perfect matching of $G$,
$M_C$ is a perfect matching of $H$. Set $\Omega=\langle\Pi,\lambda\rangle$. In what follows we will prove that $\Omega$ preserves $H$ and $\Omega/\Theta$ acts regularly on $V(H)$. We distinguish the following two cases.

\smallskip
\textsf{Case 1.} $\lambda(C)=C$.
\smallskip

In this case, we have $\lambda\in \Sigma$. Since $G_{u,\lambda}$ is a perfect matching of $G$, $\lambda$ does not fix any vertex of $C$. Hence $\lambda\notin \Theta$ and it follows that $\lambda\Theta$ is an central involution of $\Sigma/\Theta$. Since $\lambda$ is an involution fixing no vertex in $V(C)$, the length $s$ of $C$ is even. Thus, $\Aut(C)$ contains a unique central involution which is contained in any subgroup acting regularly on $V(C)$.  Since $\Pi/\Theta$ acts regularly on $V(C)$ and $\lambda\Theta$ is an central involution of $\Sigma/\Theta$, we have $\lambda\Theta\in\Pi/\Theta$. Therefore, $\Omega/\Theta=\langle\Pi,\lambda\rangle/\Theta=\Pi/\Theta$. It follows that $\Omega/\Theta$ acts regularly on $V(C)$. Since $H=(C\cup\lambda(C)) + M_C$ and $V(H)=V(C\cup\lambda(C))=V(C)$, $\Omega$ preserves $H$ and $\Omega/\Theta$ acts regularly on $V(H)$.

\smallskip
\textsf{Case 2.} $\lambda(C)\neq C$.
\smallskip

In this case, $\lambda\notin\Pi$. Since $A$ is preserved by $\Gamma$ and $C$ is a connected component of $A$, $\lambda(C)$ is also a connected component of $A$. Therefore, $\lambda(C)$ and $C$ have no common vertex as $\lambda(C)\neq C$.  Since $\lambda$ is a central involution of $\Gamma$ and $\lambda\notin\Pi$, we have $\Omega=\langle\Pi,\lambda\rangle=\Pi\times\langle\lambda\rangle$. Furthermore,  $\alpha(\lambda(v))=\alpha\lambda(v)
=\lambda\alpha(v)=\lambda(\alpha(v))$ for every $\alpha\in \Pi$ and $v\in V(C)$. Hence $\Pi$ preserves $\lambda(C)$. Moreover, since $\Pi$ acts regularly on $V(C)$, $\Pi$ acts regularly on $V(\lambda(C))$. Thus, $\Omega$ preserves $C\cup\lambda(C)$ and $\Omega/\Theta$ acts regularly on $V(C\cup\lambda(C))$. Since $H=(C\cup\lambda(C)) + M_C$ and $V(H)=V(C\cup\lambda(C))$, $\Omega$ preserves $H$ and $\Omega/\Theta$ acts regularly on $V(H)$.

Now we have proved that $\Omega$ preserves $H$ and $\Omega/\Theta$ acts regularly on $V(H)$ in either case above. So $H$ is a Cayley graph on $\Omega/\Theta$. Since $\lambda\Theta(v)$ ($=\lambda(v)$) is adjacent to $v$ for every vertex $v\in V(H)$, the Cayley multiset of this Cayley graph contains the central involution $\lambda\Theta$. Thus, by Lemma \ref{closedl}, $H$ is a closed ladder. It is obvious that the rung set of $H$ is $M_{C}$. By the arbitrariness of $C$, we obtain that every connected component of $A\cup G_{u,\lambda}$ is a closed ladder whose rung set is contained in $G_{u,\lambda}$. Similarly, every connected component of $B\cup G_{u,\lambda}$ is a closed ladder whose rung set is contained in $G_{u,\lambda}$. By Lemma \ref{pm2f}, we conclude that $G$ admits a nowhere-zero $3$-flow.
\end{proof}

\section{Proof of Theorem \ref{nil}}
\label{sec:pf}

Now we are ready to prove Theorem \ref{nil}.
Let $G$ be a nilpotently vertex-transitive graph of valency at least four. By induction on the order of $G$, we aim to prove the statement that $G$ admits a nowhere-zero $3$-flow. This statement is clearly true when $G$ is of order $2$. Now we assume that $G$ is of order greater than $2$ and the statement is true for nilpotently vertex-transitive graph of order less than the order of $G$.

We assume, without loss of generality, that $G$ is connected for otherwise we consider its components. Let $\Gamma$ be a nilpotent subgroup of $\Aut(G)$ with minimum order acting transitively on $V(G)$. By Lemma \ref{faithful}, $\Gamma$ acts faithfully on $V(G)$. Therefore, $\Gamma$ is a vertex-transitive subgroup of the automorphism group of every spanning subgraph of $G$ preserved by $\Gamma$. Since every regular graph of even valency admits a nowhere-zero $3$-flow, we assume that $G$ is of odd valency. Then $G$ is of even order. Since $G$ is $\Gamma$-vertex-transitive, the order of $G$ is a divisor of the order $|\Gamma|$ of $\Ga$. Therefore, $|\Gamma|$ is an even integer and it follows that $\Gamma$ has a nontrivial Sylow $2$-subgroup. Since a nilpotent group is a directed product of its Sylow subgroups \cite[5.2.4]{R1995}, the center of the Sylow $2$-subgroup of $\Gamma$ is contained in the center of $\Gamma$. It is well known \cite[1.6.15]{R1995} that a nontrivial $2$-group has a nontrivial center. Therefore, $\Gamma$ has a central involution, say, $\lambda$. Then $\langle \lambda\rangle$ is a normal subgroup of $\Gamma$. By Lemma \ref{free}, $\langle \lambda\rangle$ acts semiregularly on $V(G)$.

Let $u$ be a fixed vertex of $G$ and $G_{u,\lambda}$ the spanning subgraph of $G$ generated by $\lambda$ with reference vertex $u$. Set $G'=G-E(G_{u,\lambda})$ to be the graph obtained from $G$ by deleting all edges of $G_{u,\lambda}$. Since $G_{u,\lambda}$ is preserved by $\Gamma$, $G'$ is preserved by $\Gamma$. In particular, $G'$ is a $\Gamma$-vertex-transitive graph. Assume that $G_{u,\lambda}$ and $G'$ are of valency $k$ and $\ell$, respectively. Then $G$ is of valency $k+\ell$. Since $\l$ is an involution, every connected component of $G_{u,\lambda}$ is the graph with two vertices joined by $k$ parallel edges. If $k>2$, then $G_{u,\lambda}$ has a spanning cubic bipartite subgraph which is a parity subgraph of $G$ admitting a nowhere-zero $3$-flow. Therefore, $G$ admits a nowhere-zero $3$-flow. In what follows we assume $k\leq2$. Since $k+\ell$ is an odd integer greater than $4$, we have $\ell\geq3$. The rest proof is divided into three cases according to the value of $\ell$.

\smallskip
\textsf{Case 1.} $\ell=3$.
\smallskip

Since $k+\ell\geq5$ and $k\leq2$, we have $k=2$ in this case. Then every connected component of $G_{u,\lambda}$ is isomorphic to the graph with two vertices joined by two parallel edges. It is well known \cite[Theorem 16.14]{BM2008} that every $2$-edge-connected cubic graph has a perfect matching. Since $G'$ is a cubic vertex-transitive graph, every connected component of $G'$ is a $2$-edge-connected cubic graph and it follows that $G'$ has a perfect matching, say, $M$. Then every connected component of $G_{u,\lambda}\cup M$ is isomorphic to the graph constructed from a cycle of even length by doubling every edge of one of its two perfect matchings. It follows that $G_{u,\lambda}\cup M$ is a cubic bipartite graph and thus admits a nowhere-zero $3$-flow. Since $G_{u,\lambda}\cup M$ is a parity subgraph of $G$, it follows that $G$ admits a nowhere-zero $3$-flow.

\smallskip
\textsf{Case 2.} $\ell$ is an odd integer greater than $3$.
\smallskip

By the definition of $G'$, the $\langle\lambda\rangle$-orbit $\tilde{v}$ of every vertex $v\in V(G)$ is an independent set of $G'$. Thus, by Lemma \ref{efree}, $\langle\lambda\rangle$ acts semiregularly on the edge set of $G'$. Hence $\langle\lambda\rangle$ induces a regular quotient $\tilde{G'}$ of $G'$. In particular, $\tilde{G'}$ is of valency $\ell$. By Lemma \ref{vt}, $\tilde{G'}$ is $\Gamma/\Lambda^{*}$-vertex-transitive for some normal subgroup $\Lambda^{*}$ of $\Gamma$. Since $\Gamma$ is nilpotent, the quotient group $\Gamma/\Lambda^{*}$ is nilpotent. Note that the order of $\tilde{G'}$ is less than that of $G$. Thus, by the induction hypothesis, $\tilde{G'}$ admits a nowhere-zero $3$-flow. So, by Lemma \ref{lift}, $G'$ admits a nowhere-zero $3$-flow. Since $G'$ is a parity subgraph of $G$, we conclude that $G$ admits a nowhere-zero $3$-flow.

\smallskip
\textsf{Case 3.} $\ell$ is an even integer greater than $3$.
\smallskip

In this case, $G'$ is a $\Gamma$-vertex-transitive graph of valency an even integer $\ell\geq4$. Since $k+\ell$ is odd and $k\leq2$, we have $k=1$. Hence $G_{u,\lambda}$ is a perfect matching of $G$. If the stabilizer $\Gamma_{u}$ is normal in $\Gamma$, then $G$ is a Cayley graph on the quotient group $\Gamma/\Gamma_{u}$. Note that $\Gamma/\Gamma_{u}$ is also a nilpotent group. Thus, by Lemma \ref{nilCay}, $\Gamma$ admits a nowhere-zero $3$-flow when $\Gamma_{u}$ is normal in $\Gamma$.

Now we assume that $\Gamma_{u}$ is not normal in $\Gamma$. Then, by Lemma \ref{connect}, there exists $\delta\in\Gamma$ such that $\delta(u)$ is adjacent to $u$ and $\delta\Gamma_{u}\delta^{-1}\neq \Gamma_{u}$. Since $\lambda$ is in the center of $\Gamma$, we have $\delta\neq\lambda$. Consider the spanning subgraph $G_{u,\delta}$ of $G$ generated by $\delta$ with reference vertex $u$. By Lemma \ref{generating}, the number $n(G_{u,\delta},u)$ of neighbors of $u$ in $G_{u,\delta}$ is equal to $|\Gamma_u:\Gamma_u\cap\delta\Gamma_u\delta^{-1}|$ if $\Gamma_u\delta\Gamma_u=\Gamma_u\delta^{-1}\Gamma_u$ and $2|\Gamma_u:\Gamma_u\cap\delta\Gamma_u\delta^{-1}|$ if $\Gamma_u\delta\Gamma_u\neq\Gamma_u\delta^{-1}\Gamma_u$. Thus, by Lemma \ref{index}, $n(G_{u,\delta},u)$ is even. Since for every $\beta\in\Gamma_u$ the number of edges joining $u$ and $\beta\delta(u)$ is equal to that joining $u$ and $\delta(u)$, $n(G_{u,\delta},u)$ is a divisor of the valency of $u$ in $G_{u,\delta}$. Since $G_{u,\delta}$ is $\Gamma$-vertex-transitive, $G_{u,\delta}$ is a regular graph of even valency. Set $G''=G' - E(G_{u,\delta})$ to be the graph obtained from $G'$ by deleting all edges of $G_{u,\delta}$. Then $G''$ is of even valency. By Lemma \ref{stabilizer}, $\Gamma_{u}$ is contained in the Frattini subgroup $\Phi(\Gamma)$ of $\Ga$. Since $\Gamma$ is not abelian (as $\Gamma_{u}$ is not normal in $\Gamma$), we have $\Gamma\neq\langle \Gamma_{u},\delta,\lambda\rangle$. Therefore, by Lemma \ref{connect}, $G_{u,\delta}\cup G_{u,\lambda}$ is disconnected.

\smallskip
\textsf{Subcase 3.1.} $G_{u,\delta}$ is of valency greater than $2$.
\smallskip

In this subcase, $G_{u,\delta}\cup G_{u,\lambda}$ is of odd valency no less than $5$. Since $G_{u,\delta}\cup G_{u,\lambda}$ is $\Gamma$-vertex-transitive, every connected component of it is nilpotently vertex-transitive and thus admits a nowhere-zero $3$-flow by our induction hypothesis. It follows that $G_{u,\delta}\cup G_{u,\lambda}$
admits a nowhere-zero $3$-flow. Since $G_{u,\delta}\cup G_{u,\lambda}$ is a parity subgraph of $G$, we conclude that $G$ admits a nowhere-zero $3$-flow.

\smallskip
\textsf{Subcase 3.2.} $G''\cup G_{u,\lambda}$ is a disconnected graph of valency greater than $3$.
\smallskip

Similarly to Subcase 3.1, in this subcase every connected component of $G''\cup G_{u,\lambda}$ is a nilpotently vertex-transitive graph of odd valency no less than $5$ and therefore admits a nowhere-zero $3$-flow. Thus, $G$ admits a nowhere-zero $3$-flow.

\smallskip
\textsf{Subcase 3.3.} Both $G_{u,\delta}$ and $G''$ are of valency $2$.
\smallskip

In this subcase, $G_{u,\delta}$ and $G''$ are edge-disjoint $2$-factors of $G$ preserved by $\Gamma$. Recall that $\lambda$ is a central involution of $\Gamma$ and $G_{u,\lambda}$ is a perfect matching of $G$. Thus, by Lemma \ref{2factor}, $G$ admits a nowhere-zero $3$-flow.

\smallskip
\textsf{Subcase 3.4.} $G_{u,\delta}$ is of valency $2$ and $G''\cup G_{u,\lambda}$ is a connected graph of valency greater than $3$.
\smallskip

In this subcase, since $G$ is the edge-disjoint union of $G_{u,\delta}$ and $G''\cup G_{u,\lambda}$, we have that $G''\cup G_{u,\lambda}$ is of odd valency at least $5$. Note that $G''\cup G_{u,\lambda}$ is a connected $\Gamma$-vertex-transitive graph. Since $\Gamma_{u}$ is not normal in $\Gamma$, by Lemma \ref{connect}, there exists $\mu\in\Gamma$ such that $\mu(u)$ is adjacent to $u$ in $G''$ and $\mu\Gamma_{u}\mu^{-1}\neq \Gamma_{u}$. Similarly to $G_{u,\delta}\cup G_{u,\lambda}$, we have that $G_{u,\mu}\cup G_{u,\lambda}$ is a nilpotently vertex-transitive disconnected graph of odd valency. In particular, $G_{u,\mu}$ is of even valency.  If $G_{u,\mu}$ is of valency greater than $2$, then the problem boils down to Subcase 3.1. If $G_{u,\mu}$ is of valency $2$, then $G_{u,\delta}$ and $G_{u,\mu}$ are edge-disjoint $2$-factors of $G$ preserved by $\Gamma$. Since $G_{u,\lambda}$ is a perfect matching of $G$, by Lemma \ref{2factor}, $G_{u,\delta}\cup G_{u,\mu}\cup G_{u,\lambda}$ admits a nowhere-zero $3$-flow. Since $G_{u,\delta}\cup G_{u,\mu}\cup G_{u,\lambda}$ is a parity subgraph of $G$, it follows that $G$ admits a nowhere-zero $3$-flow.

By mathematical induction, we have completed the proof of Theorem \ref{nil}.

\bigskip

\noindent {\textbf{Acknowledgements}}~~The first author was supported by the China Scholarship Council (No.201808505156) and the Basic Research and Frontier Exploration Project of Chongqing (cstc2018jcyjAX0010). The second author was supported by the Research Grant Support Scheme of The University of Melbourne.

{\small

\end{document}